\documentclass[reqno,a4paper,11pt]{amsart}
\usepackage[english]{babel}
\usepackage{a4wide,xypic,ifthen}
\newcommand{\todo}[1]{\ifthenelse{\isodd{\thepage}}{\normalmarginpar}
{\reversemarginpar}\marginpar{\fbox{\parbox{24mm}{\sloppy\footnotesize #1}}}}

\usepackage{graphicx,mathrsfs,mathtools,latexsym,ifthen,amssymb,stmaryrd}

\newcommand{\Aut}{\operatorname{Aut}}

\newcommand{\Hom}{\operatorname{Hom}}
\newcommand{\End}{\operatorname{End}}

\newcommand{\Ad}{\operatorname{Ad}}

\newcommand{\rk}{\operatorname{rk}}

\newcommand{\Pic}{\operatorname{Pic}}

\newcommand{\cc}{\mathbb{C}}

\newcommand{\rr}{\mathbb{R}}
\newcommand{\zz}{\mathbb{Z}}
\newcommand{\pp}{\mathbb{P}}

\newcommand{\os}{\mathcal{O}}
\newcommand{\ox}{\mathcal{O}_X}

\newcommand{\E}{\mathcal{E}}

\newcommand{\Prb}{\mathcal{P}}

\newcommand{\xtilde}{\widetilde{X}}
\newcommand{\oxtilde}{\mathcal{O}_{\widetilde{X}}}

\theoremstyle{plain}

\newtheorem{definition}{Definition}[section]
 
\newtheorem{theorem}[definition]{Theorem} 
\newtheorem{proposition}[definition]{Proposition} 
 
\newtheorem{lemma}[definition]{Lemma}
 
\newtheorem{corollary}[definition]{Corollary}

\theoremstyle{definition}

\newtheorem{exa}[definition]{Example} 
\newenvironment{example}{\begin{exa}}{\parbox{2mm}{\hfill}\hfill $\triangle$\end{exa}}
\newtheorem{rema}[definition]{Remark} 
\newenvironment{remark}{\begin{rema}}{\parbox{2mm}{\hfill}\hfill $\triangle$\end{rema}}

\begin{document}
 
\bigskip\bigskip

\title{Restricting Higgs bundles to curves}
\bigskip
\date{\today}
\thanks{Support for this work was partly provided by {\sc prin} ``Geometry of Algebraic Varieties'' and  {\sc gnsaga-indam}. The second author is supported by the {\sc fapesp} post-doctoral grant number 2013/20617-2.   The first author is a member of the {\sc vbac} group. 
}

 \maketitle \thispagestyle{empty} \vspace{-3mm}
\begin{center}{\sc Ugo Bruzzo}$^{\P\S}$  {and} {\sc Alessio Lo Giudice}$^\ddag$ \\[5pt]
$^\P$ Scuola Internazionale Superiore di Studi Avanzati {\sc (sissa)},\\ Via  Bonomea 265, 34136
Trieste, Italia; \\[3pt] $^\S$ Istituto Nazionale di Fisica Nucleare, Sezione di Trieste \\[5pt] 
$^\ddag$ IMECC - UNICAMP,
Department of Mathematics, \\
Rua S{\'e}rgio Buarque de Holanda 651, 
Bar\~ao Geraldo, Campinas, \\ SP - Brazil
CEP 13083-859 \\[5pt] 
{\footnotesize E-mail: {\tt ugo.bruzzo@sissa.it, alessiologiudic@gmail.com} }
\end{center}

\vfill

\begin{abstract}  
We determine some classes of varieties $X$ --- that include the varieties with numerically effective tangent bundle --- satisfying the following
property: if  $\E=(E,\phi)$  is a Higgs bundle such that  $f^\ast\E$ is semistable for any morphism $f\colon C\to X$, where $C$ is a smooth projective curve, then $E$ is slope semistable and $2r c_2(E)-(r-1) c_1^2(E)=0$ in $H^4(X,\rr)$.
We also characterize some classes of varieties such that  the underlying vector bundle of a slope semistable Higgs bundle is always slope semistable.  
\end{abstract}

\maketitle

\section{Introduction}

The Chern classes of a slope semistable vector bundle $E$ on
 a polarized smooth projective variety $(X,H)$ of dimension $n$ satisfy the numerical inequality  $\Delta(E)\cdot H^{n-2}\geq0$, where the
 discriminant $\Delta(E)$ is the characteristic class
 $$\Delta(E) =c_2(E) - \frac{r-1}{2r} c_1(E)^2 \in H^4(X,\rr);$$
 here $r$ is the rank of $E$. This is called the Bogomolov inequality \cite{HL}. It may be interesting to classify the semistable bundles that satisfy the stronger equality $\Delta(E)=0$. It was proved in \cite{BHR} (see also \cite{N,BB1,BB2}) that 
 these   bundles are those whose normalized tautological divisor is numerically effective, and this condition turns out to be equivalent to the fact that for any morphism $f\colon C\to X$ from a smooth projective curve, the pullback $f^*E$ is   semistable. This critically uses the result in \cite{DPS} that numerically flat bundles have vanishing Chern classes.  In \cite{BO2} the authors proved that the pullback $f^\ast\E$ of a semistable Higgs bundle $\E=(E,\phi)$ with vanishing discriminant is semistable for any morphism $f\colon C\to X$. The other implication --- i.e.,
 if $f^\ast\E$ is semistable for all morphism $f\colon C\to X$ then $\E$ is semistable and  $\Delta(E)=0$ --- was   conjectured,  (see also \cite{SimpsonV}), however, counterexamples have been found \cite{BBG}. Here we give a proof
 of this fact for some special classes of varieties $X$.
 
 This result relies on two facts: first, for some varieties (basically, those whose tangent bundle is numerically effective), a Higgs bundle is slope semistable if and only if the underlying vector bundle is slope semistable; and then, building on this, the result is extended to a wider class of varieties by using such constructions as \'etale coverings, fibrations and the like.

The structure of this paper is as follows. In section \ref{SemistableHiggs} we give the basic definitions and main properties of $\mu$-semistable (Higgs) bundles over projective varieties. Then we generalize (Theorem \ref{TeoHs}) a result of Nitsure which allows us to relate the semistability of a Higgs bundle with the semistability of the underlying vector bundle when the base variety $X$ satisfies some conditions. This also allows one to characterize some classes of varieties such that  the underlying vector bundle of a slope semistable Higgs bundle is always slope semistable. 

In section \ref{Higgsvarieties} we give the notion of Higgs variety (Definition \ref{defHiggsvariety}) as a variety $X$ such that whenever a Higgs bundle $\E=(E,\phi)$ over $X$ is such that $f^\ast \E$ semistable for any morphism $f\colon C\to X$, where $C$ is a smooth projective curve, then the equality $\Delta(E)=0$   holds in $H^4(X,\rr)$. Using a result of Biswas and Dos Santos \cite{BD} we prove that rationally connected varieties are Higgs varieties. 
In a similar way one proves that Abelian varieties are Higgs varieties. In the second  part of this section we construct more Higgs varieties starting from the previous ones. In particular we show that finite \'etale quotients of Higgs varieties  and fibrations over Higgs varieties  with  rationally connected fibres are Higgs varieties. Since any variety with nef tangent bundle is, up to an \'etale covering,  a Fano fibration over an Abelian variety (Theorem \ref{theoremDPS}), we conclude that varieties with nef tangent bundle are Higgs varieties.

Finally we formulate our results on Higgs varieties in terms of principal Higgs bundles.

\smallskip

{\noindent\bf Acknowledgment.} We thank Alex Massarenti, Fabio Perroni, Gurjar Sudarshan and Pietro Tortella for useful discussions.  We thank the referees for their several useful comments.

\bigskip
\section{Semistable (Higgs) bundles}\label{SemistableHiggs}

Let $X$ be a smooth $n$-dimensional projective variety over the complex field and let $H$ be the numerical class of an ample line bundle on $X$. For any rank $r$ coherent sheaf $E$ we denote by $c_i(E)\in H^{2i}(X,\rr)$ its Chern classes and define the slope of $E$ (when $r>0$) as
\[
\mu(E)\colon = \frac{c_1(E)\cdot H^{n-1}}{r}.
\]
The Hilbert polynomial $P_E$ is defined as $P_E(m)=\chi(E\otimes \ox(mH))$, while
\[
\Delta(E)\colon = c_2(E)-\frac{r-1}{2r} c_1^2(E) \in H^4(X,\rr)
\] is 
the discriminant of $E$.
We recall the basic definitions of stability.

\begin{definition}
A vector bundle $E$ over the polarized variety $(X,H)$ is called stable if for all subsheaves $F\subset E$ with $0<\rk(F)<\rk(E)$ one has
\[
\frac{P_F(m)}{\rk(F)} < \frac{P_E(m)}{\rk(E)}\quad\hbox{for}\quad m\gg 0;
\]
it is called $\mu$-stable if 
\[
\mu(F)< \mu(E) .
\]
\end{definition}

The corresponding notions of semistability are obtained by replacing strict inequalities with the $\leq$ relation. 
The ratio $p_E\colon =\displaystyle\frac{P_E}{\rk(E)}$ is called the {\em reduced Hilbert polynomial} of $E$. The notion of (semi)stability given in terms of the reduced Hilbert polynomial will be sometimes called {\em Gieseker stability}, to distinguish it from $\mu$-(semi)stability when confusion might occur.

\begin{remark}
There is a chain of implications
 $\mu\textit{-stable} \Rightarrow \textit{stable}\Rightarrow  \textit{semistable}\Rightarrow \mu\operatorname{-semistable}. $
\end{remark}

\begin{remark} If $X$ is a smooth irreducible projective curve, the notions of (semi)stability and $\mu$-(semi)stability coincide. 
\end{remark}

The semistability of vector bundles is compatible with direct sums, tensor products and twists by line bundles in the following sense.

\begin{itemize} \itemsep=2pt
\item The direct sum of two ($\mu$)-semistable vector bundles is ($\mu$)-semistable if and only they have the same reduced Hilbert polynomial (slope).

\item $E$ is a ($\mu$)-semistable vector bundle if and only if $E\otimes L$ is ($\mu$)-semistable for any line bundle $L$
(this is true also for stability).

\item If $E_1$ and $E_2$ are $\mu$-semistable then $E_1\otimes E_2$ is $\mu$-semistable.

\item If $f\colon Y\to X$ is a   finite \'etale morphism of smooth projective curves, a vector bundle $F$  on $X$ is semistable    if and only if $f^*F$ is semistable.
\end{itemize}

We recall that a line bundle $L\in\Pic(X)$ is said to be numerically effective (nef) if $\deg f^*L\geq 0$ for every morphism $f\colon C\to X$ where $C$ is an irreducible smooth projective curve. A divisor $D$ is said to be numerically effective if the line bundle $\mathcal O_X(D)$ is.  A vector bundle $E$ is said to be nef if the relative hyperplane bundle $\os_{\pp(E)}(1)$ on the projective bundle $\pp(E)$ is nef. It is  called projectively flat if $\pp(E)$ is given by a projective  representation of the fundamental group of $X$
\[
\rho\colon \pi_1(X)\to \pp U(\rk(E)).
\]

\subsection{Around Bogomolov's inequality}

The characteristic class $\Delta(E)$ of a $\mu$-semistable vector bundle $E$ satisfies a numerical inequality, known as {\em Bogomolov's inequality.}

\begin{theorem}
Let $(X,H)$ be a polarized smooth complex projective variety of dimension $n\ge 2$, and $E$ a $\mu$-semistable vector bundle  on $X$. 
The inequality  
\[
 \Delta(E)\cdot H^{n-2}\geq 0
\]
holds.
\end{theorem}

The next result in a sense characterizes the vector bundles which satisfy a strong form of Bogomolov's inequality ($\Delta(E)= 0$)
as those bundles that are $\mu$-semistable after restriction to any curve in $X$. This can also be regarded as a higher-dimensional characterization of Miyaoka's criterion for semistability \cite{Mi}, and as a weak form  of the Metha-Ramanathan restriction theorem \cite{MR1}.

\begin{theorem} {\rm \cite{BHR}}\label{TeoBO}
Let $E$ be a vector bundle of rank $r$ on a polarized projective variety $(X,H)$. Then the following conditions are equivalent$\colon$ 

\begin{enumerate}\itemsep=2pt
 \item for every morphism $f\colon C\to X$, where $C$  is a smooth irreducible projective curve, $f^\ast E$ is semistable;
 
 \item $E$ is $\mu$-semistable with respect to $H$, and  $\Delta(E)=0$;

 \item $E$ admits a filtration into subsheaves
\[
0=E_0\subset E_1\subset\dots\subset E_t=E, 
\]
such that the quotients $E_i/E_{i-1}$ are projective flat bundles, and   $\mu(E_i/E_{i-1})=\mu(E)$ for all $i$.
\end{enumerate}

\end{theorem}

\begin{remark} Since condition (1) does not depend on the choice of the polarization $H$, it turns out that
 if a vector bundle $E$ with $\Delta(E)=0$ is $\mu$-semistable with respect to a polarization, then it is $\mu$-semistable with respect to {\em all} polarizations.
\end{remark}

We want to extend the previous theorem to $\mu$-semistable Higgs bundles. At the first let us recall the definition of Higgs bundle.

\begin{definition}\label{Higgsbundle}
A Higgs bundle on $X$ is a pair $\E=(E,\phi)$ consisting of a holomorphic vector bundle $E$ on $X$ and a morphism   $\phi\colon E\to E\otimes \Omega^1_X$, called the Higgs field, such that the morphism $\phi\wedge\phi\colon E\to E\otimes \Omega^2_X$ vanishes.
\end{definition}

A subsheaf $F\subset E$ is called $\phi$-invariant if $\phi(F)\subset F\otimes  \Omega^1_X$.

\begin{definition}\label{SemistabilityHiggsvectorbundle}
A Higgs bundle $\E=(E,\phi)$ is (semi)stable if for any $\phi$-invariant subsheaf $F$, with $ 0 < \rk F < \rk E$, one has
\begin{equation}
\frac{P_F(m)}{\rk(F)}    (\leq)\frac{P_E(m)}{\rk(E)} \quad\hbox{for}\quad m\gg 0;
\end{equation}
$(E,\phi)$ is called polystable if it is semistable as a Higgs bundle and it is isomorphic to a direct sum of stable Higgs bundles. 
\end{definition}

Again, a definition of $\mu$-(semi)stability can be given in the obvious way.

\begin{example}
Let $E = K^{\frac{1}{2}} \oplus K^{-\frac{1}{2}}$, where $K^{\frac{1}{2}}$ is a complex line bundle whose square is the canonical bundle of a smooth projective curve $X$.   We obtain a family of Higgs fields on $E$ parametrized by quadratic differentials, i.e., sections  of the line bundle $K^2 \simeq \mathcal \Hom(K^{-\frac{1}{2}},K^{\frac{1}{2}} \otimes K),$ by setting 
\[
\phi=
\left(
\begin{array}{cc}
0 & \omega \\
1 & 0 
\end{array}
\right)
\]
where 1 is the identity section of the trivial bundle $\mathcal \Hom(K^{\frac{1}{2}},K^{-\frac{1}{2}} \otimes K)$ and $\omega\in H^0(X,K^2)$.
Now,
$(E,\phi)$ is a stable Higgs bundle since $K^{\frac{1}{2}}$ is not $\phi$-invariant and there are no subbundles of positive degree preserved by $\phi$. However if the genus of the curve is   greater than 1,  $E$ is not semistable as a vector bundle. This example shows that the nefness of the anticanonical bundle is an obstruction to the existence of  semistable Higgs vector bundles that are not semistable in the usual sense.
\end{example}

\subsection{The case of Higgs bundles} 

For Higgs bundles conditions (2) and (3) in Theorem \ref{TeoBO} are equivalent, and both imply (1).
While in general it is not true   that (1) implies (2) or (3) (a counterexample is given in \cite{BBG}),
the implication holds for 
 some classes of varieties. In particular we will prove this for smooth projective varieties with nef tangent bundle. The idea is to connect Higgs semistability to classical semistability and thus apply Theorem \ref{TeoBO}.
We need to enlarge a little bit the class of bundles we consider. A holomorphic pair $(E,\phi)$ is a pair where $E$ is a holomorphic vector bundle on a smooth projective variety and $\phi$ is a morphism of vector bundles $\phi\colon E \to E \otimes M$ with $M$ a fixed vector bundle (thus, holomorphic pairs are special cases of {\em framed sheaves}, see \cite{HL1,HL2} for this  notion. $(E,\phi)$ is called semistable if and only if   for any $\phi$-invariant subsheaf $F$ of $E$ one has
\[
\frac{P_F(m)}{\rk(F)}\leq \frac{P_E(m)}{\rk(E)}\quad\hbox{for}\quad m\gg 0.
\]
These objects  were introduced by Nitsure in \cite{Ni}, where he studied the moduli space of semistable pairs over smooth projective curves. Here we extend some results to   varieties of any dimension.

\begin{theorem}\label{TeoHs}
Let $(X,H)$ be a polarized smooth projective variety. If $(E,\phi)$, with $\phi\colon E\to E\otimes M$, is  a $\mu$-semistable pair, and $M$ is a  $\mu$-semistable vector bundle of nonpositive degree, then $E$ is $\mu$-semistable as a vector bundle. 
\end{theorem}

\begin{proof}
Let us assume $E$ not $\mu$-semistable and consider its Harder-Narasimhan filtration \cite{HL}
\[
0\subset E_1\subset \dots \subset E_t\subset E.  
\]
If we denote by $\mu_i$ the slope of $E^i\colon = E_i/E_{i-1}$, then $\mu_1>\mu(E)$, so that $E_1$ is not $\phi$-invariant. Moreover by construction  $\mu_i >\mu_j$ if $i<j$. Let $j$ be the smallest integer such that $\phi(E_1)\subset E_j\otimes M$. Hence the homomorphism $\phi\colon E_1\to E^j\otimes M$ is not zero.
The bundle $M^j\colon= E^j\otimes M$ is $\mu$-semistable and its slope is $\mu(E^j)+\mu(M)$. By hypothesis $\deg(M)\leq 0$ hence $\mu(M^j)\leq \mu(E^j)$, however since $\Hom(E_1,M^j)\neq 0$ we get 
\[
\mu(E_1)\leq \mu(M^j)\leq \mu(E^j),
\]
and this contradicts the assumption $\mu_1>\mu_j$. So $E$ is $\mu$-semistable. 
\end{proof}
 
 \begin{corollary}\label{cor} If $\mathcal E =(E,\phi)$ is a $\mu$-semistable Higgs bundle on a smooth projective polarized variety $(X,H)$,
 whose cotangent bundle is $\mu$-semistable and has nonpositive degree, then $E$ is $\mu$-semistable.
 \end{corollary}
 
Examples of such varieties are projective spaces, Grassmannians,  a large class of Fano varieties, and  smooth projective varieties with numerically trivial canonical divisor \cite{GKP}, in particular, Calabi-Yau varieties.

\begin{remark}
In \cite{Ni} Nitsure  proved that a holomorphic pair $(E,\phi)$, where $\phi\colon E\to E\otimes L$ with $L$ a line bundle of degree zero over a smooth projective curve, is semistable if and only if $E$ is semistable as a vector bundle. It follows that the underlying vector bundle of  semistable Higgs bundle $(E,\phi)$ over an elliptic curve is semistable. Thus Corollary \ref{cor} extends this result to some classes of higher dimensional varieties.
Generalizations in other directions were given by E.~Franco G\'omez \cite{EF} (he considered the analogous result for semistability and stability of principal Higgs $G$-bundles on an elliptic curve, where $G$ is a complex linear reductive algebraic group).
\end{remark}

\bigskip
\section{Higgs varieties}\label{Higgsvarieties} 
In this section by (semi)stability we shall always mean $\mu$-(semi)stability.
Given a semistable Higgs vector bundle $(E,\phi)$ such that $\Delta(E)=0$ then  for any morphism $f\colon C\to X$, where $C$ is a smooth projective curve, the pullback $f^\ast (E,\phi)$ is semistable as a Higgs bundle. 
In this section we show that the converse result holds   for some classes of varieties.

\begin{definition}\label{defHiggsvariety}
We say that a projective variety $X$ is a Higgs variety if $\dim(X)=1$, or, in the case $\dim(X)>1$,   the following property holds: if  $\mathcal E=(E,\phi)$ is a Higgs vector bundle   on $X$  such that for any morphism $f\colon C\to X$ from a smooth projective curve $C$ the pullback $f^*\mathcal E$ is semistable as a Higgs bundle,  then 
$
\Delta(E)=0.
$
\end{definition}

We will prove that rationally connected varieties and Abelian varieties are Higgs varieties. Moreover, we shall see that finite \'etale quotients of Higgs varieties, and  fibrations over Higgs varieties with
 rationally connected fibres  are Higgs varieties. In particular this facts allows us to prove that any projective variety with nef tangent bundle is a Higgs variety.

\subsection{Rationally connected varieties}
Here we recall some general facts about rationally connected varieties. More details can be found in \cite{MP}.
\begin{definition}
A variety $X$ is   rationally connected if any two general points in $X$ are connected by a chain of rational curves.
\end{definition}
\begin{proposition}\label{ampledivisor}
If a smooth ample divisor  $D$ in a smooth projective variety  $X$ is rationally connected, then 
$X$ is   rationally connected as well.
\end{proposition} 

By Theorem \ref{TeoHs} the negativity of the canonical bundle is an obstruction to the existence of semistable Higgs bundles whose underlying vector bundle is not semistable. For vector bundles on rationally connected varieties we have the following results due to Biswas and Dos Santos \cite{BD}.

\begin{proposition}\label{biswasdossantos}
Let $X$ be a rationally connected variety. Let $E \to X$ be a vector bundle such that for every morphism $f\colon \pp^1\to X$, the pullback $f^*E$ is trivial. Then $E$ itself is trivial. 
\end{proposition}

The following result is a strengthening of Theorem \ref{TeoBO} in the case of rationally connected varieties.

\begin{corollary}
Let $E\to X$ be a vector bundle over a rationally connected variety, such that for any morphism $f\colon \pp^1\to X$ the pull back is semistable; then $E\simeq \oplus_{i=1}^r L$ where $L$ is a line bundle on $X$. 
\end{corollary}

\begin{theorem}\label{theoremrc} 
Let $\E=(E,\phi)$ be a semistable Higgs vector bundle on a rationally connected variety $X$. If for any morphism $f \colon C \to X$, where $C$ is a smooth  projective curve, the Higgs bundle $f^\ast\E$ is semistable, then $\Delta(E)=0$.
\end{theorem}

\begin{proof}
Let us consider a morphism $f\colon C\to X$ where $C$ is a rational projective curve. As by hypothesis the pullback $f^\ast\E$ is semistable as a Higgs bundle, by applying Theorem \ref{TeoHs} we get that $f^\ast E$ is semistable as a vector bundle. By the previous corollary   $E\simeq \oplus L$ and this implies   $\Delta(E)=0$.  
\end{proof}

So rationally connected varieties are Higgs varieties.

\subsection{Abelian varieties}\label{sectionAbelian}

Let $X$ be an Abelian variety. Since the tangent bundle of $X$ is trivial, its pullback via  any morphism remains trivial, hence semistable.

Let $(E,\phi)$ be a semistable Higgs bundle over $X$ and fix an ample line bundle $H$. Since $c_1(T_X)=0$, and the tangent bundle is semistable, by Theorem \ref{TeoHs} $E$ is semistable as a vector bundle. Now we want to show that if for any morphism $f\colon C \to X$ the Higgs bundle $f^\ast (E,\phi)$ is semistable  then $\Delta(E)=0$, i.e.,  $X$ is a Higgs variety.
 Let us recall that the Higgs field on $f^\ast E$ is defined by $f^\ast\phi\colon f^\ast E\to f^\ast E\otimes K_C$ composing the pull-back with the projection $f^\ast\colon \Omega^1_X\to K_C$ induced by $f$. 
For any such morphism $f$ one can consider the pair $(f^\ast E,\phi')$ where 
\[
\phi'\colon f^\ast E\to f^\ast E\otimes f^*\Omega^1_X.
\]
Clearly a $\phi'$-invariant subbundle $F\subset E$ is also $f^\ast\phi$-invariant. In particular we have

\begin{lemma}\label{lemmaAbelian}
If the Higgs bundle $(f^\ast E,f^\ast\phi)$ is semistable,  so is the pair $(f^\ast E,\phi')$.
\end{lemma}

\begin{proof}
Let $F$ be a $\phi'$-invariant subbundle of $f^*E$, then since it is $f^\ast\phi$-invariant and the Higgs bundle $(f^\ast E,f^\ast\phi)$ is semistable, we have 
\[
\mu(F)\leq \mu(f^\ast E),
\]
and so  $(f^\ast E,\phi')$ is semistable as a pair.
\end{proof}

\begin{corollary}\label{abelianHiggs}
An Abelian variety $X$ is a Higgs variety. 
\end{corollary}

\begin{proof}
Let $\E=(E,\phi)$ be a Higgs bundle and $f\colon C\to X$ be a morphism from a smooth projective curve. Assume that the Higgs bundle $(f^\ast E,f^\ast\phi)$ is semistable; then the pair $(f^\ast E,\phi')$ is semistable. Since $f^*(\Omega^1_X)$ is semistable of degree zero, one can apply Theorem \ref{TeoHs} and conclude that $f^\ast E$ is semistable. Hence by Theorem \ref{TeoBO} one gets $\Delta(E)=0$.
\end{proof}

The condition about the semistability of the tangent bundle holds also with the weaker assumption that $X$ is a quasi-Abelian variety, moreover P. Jahnke and I. Radloff (\cite{JR}) proved that quasi-Abelian varieties are the only ones which satisfy the condition that the pull-back of tangent bundle is semistable for any morphism from a smooth projective curve. So Corollary \ref{abelianHiggs} holds also for quasi-Abelian variety; however we will prove this fact in a more general context in Proposition \ref{propositionfinitequot}.

\subsection{More Higgs varieties}\label{moreHiggs}

We want to produce new examples of Higgs varieties starting from those we have so far discussed. The first technique is to use the Lefschetz hyperplane theorem.

\begin{proposition}
Let $(X,H)$ be a smooth polarized projective variety with $\dim X=n\geq 5$ and let $D$ be a smooth effective ample divisor in $X$. If $D$ is a Higgs variety, then   $X$ is a Higgs variety as well.
\end{proposition}

\begin{proof}
Let $\mathcal E$ be a Higgs vector bundle on $X$ such that for any $f\colon C\to X$, $f^\ast \mathcal E$ is semistable as Higgs bundle. Replacing $E$ by $\End(E)$ we can assume   $c_1(E)=0$. Let us consider the vector bundle $E|_D$. For any morphism $g\colon C\to D$, the pullback of $E|_D$ is  semistable. Since $D$ is a Higgs variety we have $\Delta(E|_D)=c_2(E|_D)=0$. As by the Lefschetz theorem the morphism 
\[
H^i(X,\cc) \to H^i(D,\cc)  
\]
is injective for $i\ \le n-1$, we have $\Delta(E)=c_2(E)=0$, and so $X$ is a Higgs variety.
\end{proof}

\begin{remark}
Proposition  \ref{ampledivisor} tells us that if $X$ has a rationally connected ample divisor then $X$ is rationally connected so in this case the previous result does not give anything new.
\end{remark}

\begin{proposition}\label{corollaryrcfibration}
Let $X$ be a Higgs variety, and assume that there is a surjective morphism $g\colon Y\to X$ such that each fibre $Y_x$ is rationally connected. Then $Y$ is a Higgs variety.
\end{proposition}

\begin{proof}
Let $\mathcal E=(E,\phi)$ be a Higgs bundle on $Y$ such that $f'^*\mathcal E$ is Higgs semistable for any $f'\colon C'\to Y$. Replacing $\mathcal E$ by $\mathcal E\otimes \mathcal E^\vee$ we can assume that $c_1(E)=0$. Let $h\colon\mathbb P^1\to Y_x$ be a morphism. By hypothesis, $h^\ast \mathcal E$ is semistable. Then it is not hard to see, by inspecting the Harder-Narasimhan filtration of the underlying bundle $h^\ast E$, that the latter is semistable in the usual sense, hence it is trivial. Then by
 Proposition \ref{biswasdossantos}, $E$ is trivial on the fibres $Y_x$, and so the Higgs field of $\mathcal E$ on a fibre  $Y_x$ is a collection of global holomorphic 1-forms on $Y_x$. As this variety is rationally connected,
 there are no such holomorphic forms, and therefore the restriction of $\mathcal E$ on each fibre is the trivial Higgs bundle. Then  we have    $\mathcal E=g^* \mathcal{F}$ for some  Higgs vector bundle $\mathcal F=(F,\varphi)$ on $X$.  
 
 Let $f\colon C\to X$ be a morphism with $C$ a smooth projective curve. We have the following commutative diagram
\begin{equation}\label{diag1}
\xymatrix{ C \times_X Y  \ar[d]_{\bar{g}}\ar[r]^{\bar{f}} & Y \ar[d]^g \\
C \ar[r]^(0.5){f} & X \ .
}
\end{equation}
We claim that $\bar{f}^*\mathcal E$ is semistable. Indeed, let $\tilde f \colon \tilde C \to C \times_X Y$ be any morphism,
with $\tilde C$ a smooth projective curve. Then $\tilde f^\ast (\bar f^\ast \mathcal E) = (\bar f\circ\tilde f)^\ast \mathcal E$ is
semistable. Since this is true for any morphism $\tilde f \colon \tilde C \to C \times_X Y$, it follows that $\bar{f}^*\mathcal E$ is semistable.

Now, $\bar{f}^*\mathcal E=(g\circ\bar{f})^*(\mathcal F)$, and  $g\circ \bar{f}=f\circ\bar{g}$, so that  $(f\circ\bar{g})^*(\mathcal F)$ is semistable as a Higgs bundle. Hence also $f^*(\mathcal F)$ is. Since $X$ is a Higgs variety we get $\Delta(F)=0$ which clearly implies $\Delta(g^*(F))=\Delta(E)=0$ and $Y$ is a Higgs variety.  \end{proof}

In particular the previous Proposition  implies that ruled surfaces are Higgs varieties.

\begin{proposition}\label{propositionfinitequot}
Let $g\colon Y\to X$ be a finite \'etale map between smooth projective varieties. If $Y$ is a Higgs variety then also $X$ is so.  
\end{proposition}
\begin{proof}
Let $\E=(E,\phi)$ be a Higgs vector bundle on $X$ such that for any morphism $f\colon C\to X$ from a smooth projective curve $C$ the pullback $f^\ast \E$ is semistable as a Higgs bundle.
Let $h\colon C'\to Y$ be any morphism, we can consider the composition $g\circ h\colon C'\to X$ and so we get $h^*(g^*\E)=(g\circ h)^*\E$ is semistable as Higgs bundle, hence, since $Y$ is a Higgs variety, $\Delta(g^*E)=0$. Our hypothesis on $g$ tells us that the morphism $g^*$ is injective in cohomology;  in particular $\Delta(E)=0$ and we are done.
\end{proof}

\begin{proposition}
Let $X$ and $Y$ be smooth surfaces and $g\colon X\to Y$ be a birational map which is an isomorphism between big open subsets of $X$ and $Y$. Then $X$ is a Higgs variety if and only if $Y$ is so. 
\end{proposition}

\begin{proof}
Let $\E=(E,\phi)$ be a Higgs bundle on $Y$. Then $g^\ast E$ is a vector bundle on a big open subset $U$ of $X$, hence it can be extended uniquely to a vector bundle on all $X$ (``big'' means that the complement of $U$ has codimension at least 2). Moreover $g$ induces a one-to-one correspondence between $Mor(C,Y)$ and $Mor(C,X)$, where $C$ is a smooth curve, given by 
\[
(f\colon C\to Y) \longmapsto (g^{-1}\circ f\colon C\to X), 
\]
since birational maps between curves are actually isomorphisms we get that the map $g^{-1}\circ f$ is actually a  morphism. Thus  $f^\ast \E=(g^{-1}\circ f)^*(g^*(\E))$. Since $X$ is a Higgs variety we obtain $\Delta(\phi^*E)=0$ hence $\Delta(E)=0$ and $Y$ is a Higgs variety.   
\end{proof}

\subsection{Varieties with nef tangent bundle}

Let us observe that all varieties we so far described in this section have nef tangent bundle. Indeed one can
 show that varieties with nef tangent bundle are Higgs varieties. 
The main theorem we use is the following result   by Demailly, Peternell and Schneider \cite{DPS}.

\begin{theorem}\label{theoremDPS}
Let $X$ be a compact K\"ahler manifold with nef tangent bundle $T_X$. Let $\xtilde$ be a finite \'etale cover of $X$ of maximum irregularity $q =q(\xtilde) = h^1(\xtilde ,\oxtilde)$. Then
\begin{enumerate}
\item $\pi_1(\xtilde) \simeq \zz^{2q}.$

\item The Albanese map $\alpha \colon \xtilde \to A(\xtilde)$ is a smooth fibration over a $q$-dimensional torus with nef relative tangent bundle.

\item The fibres $F$ of $\alpha$ are Fano manifolds with nef tangent bundles. 
\end{enumerate}
\end{theorem}

\begin{corollary}
Any projective variety with nef tangent bundle is a Higgs variety. 
\end{corollary}

\begin{proof}
Thanks to Proposition \ref{propositionfinitequot} we can study Higgs varieties up to finite 
\'etale cover. So we can assume that $X$ satisfies the condition in the previous theorem. In particular since smooth Fano varieties are rationally connected varieties and Abelian varieties are Higgs varieties, the thesis follows from Proposition \ref{corollaryrcfibration}. 
\end{proof}

\subsection{Principal Higgs bundles}

The results we have found in this paper can be easily transferred to Higgs principal bundles.
Let $G$ be a complex reductive linear algebraic group, and let $\Ad:G\to \Aut(\mathfrak{g})$ be the adjoint representation of $G$ into its Lie algebra $\mathfrak{g}$. The vector bundle associated with a principal $G$-bundle $P$ by the adjoint representation will be denoted $\Ad P$; there is a natural bracket defined on its sections.
\begin{definition} A Higgs principal $G$-bundle over $X$ is a pair $\Prb=(P,\phi)$, where $P$ is a principal $G$-bundle over $X$, and $\phi$ is a global section  of $\Ad(P)\otimes \Omega^1_X$ such that $[\phi,\phi]=0$.
\end{definition} 

If $\Prb = (P,\phi )$ is a principal Higgs $G$-bundle, the adjoint bundle $\Ad(P)$ has vanishing first Chern class, and is semistable, as a Higgs vector bundle, if and only if $\Prb$ is semistable as a principal Higgs bundle (for the definition of semistabiliy of a Higgs principal bundle see \cite{BO2}). Thus we get:
\begin{corollary}
Let $X$  be a   Higgs variety. If for any morphism $f \colon C \to X$, where $C$ is a smooth projective curve, the principal Higgs bundle $f^*\Prb$ is semistable, then $\Prb$ is semistable for any polarization on $X$, and $c_2(\Ad(P)) = 0$.
\end{corollary}
This completes the result of \cite{BO2} in the case of Higgs varieties.

\begin{remark} Many results of this section also hold for the wider class of holomorphic pairs (as opposed to Higgs bundles).
\end{remark}

\end{document}